\documentclass[a4paper,11pt]{amsart} %{article}
\usepackage[left=3.3cm, right=3.3cm]{geometry}
\usepackage[T1]{fontenc} 
\usepackage{mathpazo}
\usepackage{euler}

\usepackage{setspace}
\usepackage[dvipsnames]{xcolor}

\usepackage{tikz}
\usepackage{tikz-cd}

\usepackage{mathtools}

\usepackage{mathrsfs}
\usepackage{graphicx}
\usepackage{enumerate} 
\usepackage{cancel}
\usepackage{latexsym}
\usepackage{amssymb,hyperref}
\hypersetup{colorlinks,linkcolor={RoyalBlue},citecolor={Orchid},urlcolor={OliveGreen}}
\usepackage[cp850]{inputenc}
\usepackage{epsfig}
\usepackage{amsthm}
\usepackage{amscd}
\usepackage{amsmath}
\usepackage{amsfonts}
\usepackage{graphics}
\newtheorem{theorem}{Theorem}[section]

\newtheorem{lemma}[theorem]{Lemma}
\newtheorem{corollary}[theorem]{Corollary} 
\newtheorem{prop}[theorem]{Proposition}
\newtheorem{notation}[theorem]{Notation}

\newtheorem*{theorem*}{Theorem} 
\newtheorem*{corollary*}{Corollary}
\theoremstyle{definition}
\newtheorem{example}[theorem]{Example}
\newtheorem{remark}[theorem]{Remark}
\newtheorem{definition}[theorem]{Definition}
\newtheorem*{remark*}{Remark}

\newtheorem*{definition*}{Definition}
\newtheorem*{example*}{Example}

\newtheoremstyle{named}{}{}{\itshape}{}{\bfseries}{.}{.5em}{\thmnote{#3}}
\theoremstyle{named}

\newtheoremstyle{named}{}{}{\itshape}{}{\bfseries}{.}{.5em}{\thmnote{#3}}
\theoremstyle{named}

 \newcommand{\BG}{\mathbb G}

\newcommand\smvee{\raise0.3ex\hbox{$\scriptscriptstyle\vee$}}

\DeclareMathOperator{\Spec}{Spec}

\newcommand{\comment}[1]{}

\DeclareSymbolFont{greekletters}{OML}{cmr}{m}{it}
\DeclareMathSymbol{\vsigma}{\mathalpha}{greekletters}{"26}

\title{Log purity, torsors on root stacks and log Nori fundamental group}

\author{Sara Mehidi}
\address{Mathematical Institute, Utrecht University, Hans Freudenthalgebouw, Budapestlaan 6, 3584 CD, Utrecht, Netherlands}
\email{s.mehidi@uu.nl}

\begin{document}

 \begin{abstract}   We generalize the logarithmic purity theorem of \cite{FK,Mochizuki} to torsors which arise in the Kummer log flat topology under finite flat linearly reductive group schemes. We then give a stack-theoretic interpretation of our purity theorem via root stacks, relating it to the valuative criterion of properness for tame algebraic stacks \cite{root}. Finally, we construct a logarithmic Nori fundamental group scheme of a log regular log scheme classifying such torsors, and compare it with the classical Nori fundamental group and the tame fundamental group.
\end{abstract}
\maketitle
\setcounter{tocdepth}{2}
\tableofcontents
\section{Introduction}
\addtocontents{toc}{\protect\setcounter{tocdepth}{1}}
The \textit{Purity Theorem} of Zariski--Nagata asserts that a finite flat morphism $f \colon Y \to X$ onto a regular scheme, which is \'etale over a dense open subset $U \subseteq X$ containing all codimension one points of $X$, is in fact \'etale everywhere. A similar statement in positive characteristic was claimed and proved in \cite{moreb,marrama}, where finite \'etale covers are replaced by torsors under finite flat group schemes (not necessarily \'etale). More precisely, they proved that the restriction functor induces an equivalence between the category of fppf torsors on $X$ and the category of fppf torsors on $U$.\\

However, in many natural situations, one is led to consider a scheme $X$ (possibly singular) together with a dense open subset $U \subseteq X$ whose complement is a divisor, or more generally contains codimension one points. Typical examples include:
\begin{enumerate}
\item $U$ a quasi-projective variety and $X$ a toroidal or simple normal crossings compactification,
\item $U$ a family of smooth curves over some base, degenerating to a family of nodal curves $X$.
\end{enumerate}

The classical purity theorem does not capture situations involving ramification along a boundary divisor. Logarithmic geometry, a variant of algebraic geometry introduced by Fontaine-Illusie in the 1980s, and further developed by Kato and many others, provides a more flexible framework for studying such phenomena. In particular, it allows one to reinterpret tame ramification as \'etaleness in the log setting.

Indeed, in characteristic zero, \cite{FK,Mochizuki} established a logarithmic analogue of the purity theorem (cf. Theorem \ref{logetpur}). More precisely, they proved that finite \'etale covers of $U$ that are tamely ramified along the boundary extend uniquely to Kummer log \'etale covers of $X$.\\

The first result of this paper extends the logarithmic purity theorem to torsors under finite flat linearly reductive group schemes. In positive characteristic, such torsors provide a natural enlargement of the class of tamely ramified covers (cf. \cite{AbrOV}) and account for ramification phenomena that are invisible to the logarithmic \'etale topology.

\begin{theorem}[Theorem~\ref{lin-red}]\label{A}
Let $X$ be a log regular log scheme, and let $U \subseteq X$ be the dense open locus where the log structure is trivial. Let $G$ be a finite flat linearly reductive group scheme over $X$. Then the restriction functor
\[
\mathrm{Tor}_{\mathrm{klf}}(X,G) \longrightarrow \mathrm{Tor}_{\mathrm{fppf}}(U,G)
\]
is an equivalence of categories.
\end{theorem}

In examples (a) and (b) above, the base $X$ admits a natural logarithmic structure making it a log regular log scheme for which Theorem \ref{A} applies.\\

The problem of extending torsors in the logarithmic setting has been studied in \cite{Sara}, \cite{Saraa}, and \cite{saraT}, specifically for torsors on curves. In the latter reference it is shown, using tropical geometry, that when $X \to S$ is a family of nodal curves endowed with a logarithmic structure making it log smooth, every fppf $G$-torsor on the smooth locus of $X$ extends uniquely to a logarithmic $G$-torsor on $X$ when $G$ is diagonalizable. In Corollary~\ref{logcurve}, we extend this result to the case where $G$ is linearly reductive.\\

The main ingredient in the proof of Theorem \ref{A} is a structural description of logarithmic torsors under finite flat linearly reductive group schemes. Namely, such torsors can be decomposed into the composition of an \'etale torsor and a torsor under a diagonalizable group scheme. We then prove that each of these components extends separately in the logarithmic setting, and that these extensions combine to yield an extension of the original torsor.

\subsection{Log torsors and root stacks}

The relationship between logarithmic geometry and its stack-theoretic incarnation has been studied in \cite{BorneVis,TalpoV,Olssonn}. A fundamental observation due to Olsson is that Kummer log \'etale morphisms over a base become \'etale after passing to an appropriate root stack of the base. The notion of root stack used here is that of a stack associated to a logarithmic scheme obtained by modifying the monoids of its log structure, as introduced in \cite{BorneVis}. When the logarithmic structure is induced by an irreducible effective Cartier divisor, this construction coincides with the classical root stacks of \cite{Cadman}.\\

More generally, the infinite root stack of a log scheme was introduced in \cite{TalpoV} as a way to encode the full logarithmic structure by simultaneously taking all finite-level root stacks. One of the main results of \cite{TalpoV} establishes an equivalence
\[\mathrm{Sh(X_{klf})} \equiv \mathrm{Sh}((\sqrt[\infty]{X})_{fppf})\]
between the Kummer flat topos of a logarithmic scheme $X$ and the fppf topos of its infinite root stack $\sqrt[\infty]{X}:=\varprojlim\limits_{\substack{n}}\sqrt[n]{X}$.\\
From this, it is deduced in \cite{Talpp} that if $X$ is a $k$-scheme and  $\sqrt[\infty]{X}' := \varprojlim\limits_{\substack{n \\ \mathrm{char}(k)\nmid n}}\sqrt[n]{X}$, one has
\[
\mathrm{fk\acute{e}t}(X)
\equiv
\mathrm{f\acute{e}t}(\sqrt[\infty]{X}')
\]
where the category on the left is that of finite Kummer log \'etale covers, and the one on the right is that of finite \'etale covers.\\

In this paper, we extend this equivalence to the positive characteristic, i.e. taking into account all finite flat torsors beyond the unramified ones. We prove:

\begin{theorem}(Theorem \ref{equi})\label{B}
Let $X$ be a log scheme, and let $G$ be a finite flat group scheme over $X$. Then there is an equivalence of categories
\[
\mathrm{Tor}_{klf}(X,G) \equiv \mathrm{Tor}_{fppf}(\sqrt[\infty]{X},G_{}).
\]
%where $\sqrt[\infty]{X}' := \varinjlim\limits_{\substack{n}}\sqrt[n]{X}$ is the infinite root stack of \cite{TalpoV}. 
\end{theorem}
\hfill

Since $G$ is finite, every such torsor descends to a finite-level root stack. Consequently, the problem of extending an fppf torsor on the trivial locus of a logarithmic scheme admits a natural stack-theoretic interpretation: it is equivalent to extending it to some root stack of the base. \\

Combined with Theorem \ref{A}, this comparison theorem allows us to observe an interesting connection with the valuative criterion of properness for proper tame algebraic stacks \cite{root}. Indeed, proper tame algebraic stacks are \'etale-locally quotient stacks of the form $[Y/G]$, where $G$ is a finite flat linearly reductive group scheme. In particular, the classifying stack $BG$ may be viewed as the basic local model for such stacks. This point of view leads to the following corollary.

\begin{corollary}(Corollary \ref{cor})
Let $X$ be either 
\begin{enumerate}
    \item a toric variety with dense torus $U$ and toric divisor $D$,
    \item a nodal proper curve with $U$ the smooth locus and $D:=X\backslash U$.
\end{enumerate}
Let $G$ be a finite flat linearly reductive group scheme over $X$. Then any rational map \begin{tikzcd}[column sep=0.8em]
X \arrow[dashed]{r} & BG
\end{tikzcd} extends (in a unique way) to a morphism $\sqrt[n]{X} \to BG$ for some integer $n$. In addition, if the underlying scheme of $X$ is regular and $D$ is a strict normal crossing divisor, then for every point $x \in D$, the induced fppf $G$-torsor on $U$ specializes to an fppf $G$-torsor on $\Spec k(x)$.
\end{corollary}

\subsection{Logarithmic Nori fundamental group}
Grothendieck's classical \'etale fundamental group has been extended in two different directions in order to capture broader classes of coverings. In positive characteristic, the Nori fundamental group scheme \cite{Nori} refines Grothendieck's construction by classifying torsors under finite flat group schemes, including infinitesimal ones. On the other hand, the logarithmic \'etale fundamental group \cite{FK,illusie2002overview} provides a logarithmic analogue of the \'etale fundamental group and, under suitable regularity assumptions, recovers the tame fundamental group of Grothendieck-Murre \cite{tame-fundamental}. Consequently, it classifies tamely ramified covers in the classical sense, namely finite generically \'etale covers with tame ramification along the boundary.\\

In positive characteristic, however, the log \'etale fundamental group does not capture all tame phenomena: it misses torsors under finite linearly reductive group \linebreak  schemes, whose ramification behavior should still be regarded as tame in the sense of \cite{AbrOV}. As another application of Theorem \ref{A}, we construct a logarithmic Nori fundamental group which classifies precisely those logarithmic torsors. Our construction is the logarithmic version of the maximal linearly reductive quotient of the classical Nori fundamental group scheme studied in \cite{Otabe}. Let $k$ be a field of characteristic $p > 0$. We have:

%and let $X$ be a log regular log scheme over $k$ with integral underlying scheme. Fix a $k$-log point $x_0 \to X$, i.e. the underlying scheme of $x_0$ is $\Spec k$ and it is endowed with a log structure inducing a morphism of log schemes $x_0 \to X$. We consider the category of Kummer log flat torsors over $X$ pointed over $x_0$ and under finite linearly reductive $k$-group schemes. Using Theorem~\ref{A}, we show that this category is cofiltered and admits a universal pro-object. This leads to:

\medskip

\begin{theorem} (Theorem~\ref{pi} and Corollary \ref{U})
Let $X$ be a log regular $k$-log scheme with integral underlying scheme and let $x_0$ be a $k$-log point. Then there exists a profinite pro-linearly reductive $k$-group scheme
\[
\pi^{N,\log}_{1}(X,x_0)
\]
classifying Kummer log flat torsors over $X$ (pointed over $x_0$) under finite linearly reductive group schemes. In addition, if $x_0$ has the trivial log structure, then
\[\pi^{N, \log}_{1}(X,x_0) \simeq \pi^{N}_{1}(U,x_0)\]
where $\pi_1^N$ denotes the maximal pro-linearly reductive quotient of the (classical) Nori fundamental group on $U$.
\end{theorem}

%Naturally, when the log structures are all trivial, this recovers the maximal pro-linearly reductive quotient of the (classical) Nori fundamental group scheme of $X$. If $k$ is algebraically closed of characteristic $0$, then finite linearly reductive group schemes coincide with constant finite group schemes. Hence, in this case, the construction recovers the logarithmic \'etale fundamental group.\\

\textbf{Acknowledgments.}
The author thanks Enhao Feng, Jean Gillibert and Boaz Moremann for helpful discussions and comments,  Giulio Bresciani for helpful answers to questions about their paper, and Marta Pieropan for a feedback on a previous version. The author was supported by the NWO Vidi grant VI.Vidi.213.019. For the purpose of open access, a CC BY public copyright license is applied to any Author Accepted Manuscript version arising from this submission.

\section{Log torsors}
\addtocontents{toc}{\protect\setcounter{tocdepth}{2}}
\begin{definition}[Log scheme]
    A log scheme is the data of a scheme $X$ together with a sheaf of monoids $\mathcal{M}_X$ in the small \'etale site of ${X}$ with a morphism of sheaves of monoids $\alpha_X: \mathcal{M}_X \to \mathcal{O}_X$ such that $\alpha_X^{-1}(\mathcal O_X^\times)\xrightarrow{\sim}\mathcal O_X^\times$, where $\mathcal{O}_X$ is viewed as a monoid with the multiplicative law. The sheaf $\mathcal{M}_X$ is called a log structure. 
\end{definition}
\begin{notation}
    When $X$ is a log scheme, we denote $\underline{X}$ its underlying scheme.
\end{notation}

\begin{example}\label{log str}
  \begin{enumerate}
      \item If ${X}$ is a scheme, setting $\mathcal{M}_X:=\mathcal{O}_X^{\times}$ defines the trivial log structure, which allows to see any scheme as a log scheme.
      \item If ${X}$ is either a regular scheme or a toric variety and $j:U \hookrightarrow {X}$ is either a dense open subset which is the complement of normal crossing divisor $D$, or the dense torus of a toric variety, the sheaf of monoids $\mathcal{M}_X:=j_*\mathcal{O}_U^{\times} \cap \mathcal{O}_X$ defines a log structure on ${X}$ called the divisorial log structure.
      \item Let $P$ be a finitely generated monoid. The inclusion of monoids $P \subset \mathbb{Z}[P]$  induces a canonical log structure on the scheme $\Spec \mathbb{Z}[P]$ (cf.  \cite[Proposition I. 1.1.5]{Ogus2018Lectures}).
  \end{enumerate}
\end{example}

All the log schemes considered here are fine and saturated, i.e. the stalks of the log structure are fine (finitely generated and integral) and saturated monoids. Given a log scheme $X$, we denote by $(\mathrm{fs})$ (resp. $(\mathrm{fs}/X)$) the category of fine saturated log schemes (resp. fine saturated log schemes over $X$). Fine and saturated log schemes admit local charts by fine and saturated monoids: for every geometric point $x$ in the underlying scheme of $X$, there exists an \'etale neighborhood $U$ of $x$ and a morphism of log schemes $f:U \to \Spec \mathbb{Z}[P]$ which is strict, i.e. the log structure $\mathcal{M}_U:=\mathcal{M}_{X|U}$ of $U$ is isomorphic to the canonical log structure on $U$ associated to $f^{-1}\mathcal{M}_{\Spec \mathbb{Z}[P]} \to f^{-1}\mathcal{O}_{\Spec \mathbb{Z}[P]} \to \mathcal{O}_U$ (cf. \cite[Definition III.1.1.5]{Ogus2018Lectures}).\\

In particular, all fiber products of log schemes considered here are understood in the category of fs log schemes (which is in general different from the classical fiber product, see Example \ref{fiber prod}). We refer to \cite{Ogus2018Lectures} for a more comprehensive introduction to log geometry. We briefly recall the definitions of the Kummer log flat topology and Kummer log flat torsors and refer to \cite[Section 2]{katoLogStructuresII} for further details. \\

Log flat, log smooth and log \'etale morphisms of log schemes can be defined in terms of local charts of the log structures \cite[Definitions IV.3.1.1 and  IV.4.1.1]{Ogus2018Lectures} \cite[Theorem 4.1]{Kato1996Log-smooth-defo}. These notions naturally generalize their classical counterparts in algebraic geometry: a strict and flat (resp. smooth, resp. \'etale) morphism is log flat (resp. log smooth, resp. log \'etale). On the other hand, the diagonal morphism $\mathbb A^1 \longrightarrow \mathbb A^1\times \mathbb A^1$ between toric varieties endowed with their divisorial log structures is an example of a log flat morphism which is not flat. Likewise, a nodal curve over a point can be endowed with logarithmic structures making it log smooth over the point. Tamely ramified covers in the sense of \cite{tame-fundamental} and toric blow-ups provide examples of log \'etale morphisms. Finally, Kummerness is an additional condition on the morphism of log structures, which can also be described on the local charts \cite[Definition III.2.3.2]{Ogus2018Lectures}.

\begin{definition}[klf topology]
Let $X$ be a log scheme. A family of morphisms $\{f_i : X_i \to X\}_i$ of log schemes is called a covering of $X$ for the Kummer log flat (resp. Kummer log \'etale) topology if the conditions (1) (resp. (1)$'$) and (2) are satisfied:
\begin{enumerate}
    \item[(1)] The morphisms $f_i$ are log flat and of Kummer type, and the underlying morphisms of schemes are locally of finite presentation.
    \item[(1)$'$] The morphisms $f_i$ are log \'etale and of Kummer type.
    \item[(2)] $X = \bigcup_i f_i(X_i)$ (set-theoretically).
\end{enumerate}

We obtain a Grothendieck topology, called the Kummer log flat (resp. Kummer log \'etale) topology, denoted klf (resp. k\'et) for simplicity, on $(\mathrm{fs}/X)$ by taking coverings as above, and we use the induced site to define log torsors.
\end{definition}

\begin{definition}[Logarithmic torsor]
Let $Y$, $X$ be log schemes and $G$ a group scheme over $X$ such that $G \to X$ is strict. Assume that $G$ acts on $Y$ over $X$ in the category of log schemes. We say that $Y \to X$ is a klf $G$-torsor (or simply a log $G$-torsor) if there exists a klf covering $X_i \to X$ and $G$-equivariant isomorphisms $Y \times_X X_i \simeq G_{X_i}$ over $X_i$, where $G_{X_i}$ is equipped with the left translation action of itself. 
\end{definition}

\begin{example}\label{fiber prod} Let $R$ be a discrete valuation ring with uniformizer $\pi$. Let $X = \Spec R$ with divisorial log structure induced by the closed point. It admits a chart
\[
\mathbb N \longrightarrow R,
\qquad
1 \longmapsto \pi .
\]

Set
\[
Y = \Spec R[T]/(T^r - \pi)
\]
with log structure induced by the chart
\[
\mathbb N \longrightarrow R[T]/(T^r-\pi),
\qquad
1 \longmapsto T .
\]
The morphism of log schemes $Y \to X$ is induced by the Kummer map of monoids
\[
\mathbb N \xrightarrow{\times r} \mathbb N.
\]

There is a natural action of $\mu_r$ on $Y$ (which is not free on the scheme-theoretic level) making $Y \to X$ a klf $\mu_r$-torsor. If $r$ is invertible in $R$, then it is a Kummer log \'etale $\mu_r$-torsor. In particular, the fs fiber product $Y \times_X Y$ is isomorphic to $Y \times \mu_r$. 
\end{example}

\section{Log Purity Theorem}
\subsection{Moduli of torsors on toric varieties}
The purpose of this section is to extend the classical purity theorem to the logarithmic setting (Theorem \ref{lin-red}). Such a version has already been established in characteristic zero, identifying generically \'etale covers that are tamely ramified along the boundary with logarithmically \'etale covers. Our aim here is to also treat the positive characteristic case, where additional tamely ramified covers may occur. We assume in this section that our schemes are locally noetherian. We recall the statement of the classical purity theorem.

\begin{theorem}[Zariski-Nagata purity]
Let $X$ be a regular scheme and $U \subset X$ a dense open subset containing all codimension one points of $X$. Let
$f:Y \to X$ be a finite morphism, which is \'etale over $U$. Then $f$ is \'etale everywhere.
\end{theorem}
\begin{example}\label{tam}
   Let $f : E \xrightarrow{2:1} \mathbb{P}^1_{\mathbb{C}}$, $(x,y)\mapsto x$, where $E$ is the elliptic curve defined by $y^2=(x-1)(x-2)(x-3)$. Then $f$ is a finite generically \'etale morphism of degree $2$, ramified exactly at the three (codimension one) points $x=1,2,3$. Moreover, the group $\mathbb{Z}/2\mathbb{Z}$ acts on $E$ via $(x,y)\mapsto (x,-y)$, and this action is free away from the ramification locus of $f$.
\end{example}

Example \ref{tam} illustrates a case of an \'etale $\mathbb{Z}/2\mathbb{Z}$-torsor over a dense open subset of $\mathbb{P}^1_{\mathbb{C}}$ which acquires ramification at codimension one points when extended to the entire projective line.
The classical purity theorem does not cover situations such as in Example~\ref{tam}. A logarithmic version has since been established, taking into account the case where $U$ is a dense open subset whose complement is a divisor.\\

We refer to \cite[Definition 2.1]{KatoToric} for the definition of a log regular log scheme. In particular, important examples include Examples~\ref{log str}(b). Furthermore, if $X$ is log regular, then the locus where the log structure is trivial is a dense open subset.

\begin{theorem}\cite{FK,Mochizuki}\label{logetpur}
Let $X$ be a log regular log scheme, and let $U \subset X$ be the dense open subset where the log structure is trivial. Then the restriction functor
\[
\mathrm{fk\acute{e}t}(X) \longrightarrow \mathrm{f\acute{e}t}(U)
\]
from the category of finite Kummer log \'etale covers of $X$ to the category of finite \'etale covers of $U$ is fully faithful. Its essential image consists of those finite \'etale covers of $U$ which are tamely ramified along $X \setminus U$ in the sense of Grothendieck--Murre (cf. \cite{tame-fundamental}).
\end{theorem}

In particular, one can deduce the following corollary:
\begin{corollary}\label{inv}
    Let $X$ be a log regular log scheme and $U \subset X$ the dense open subset where the log structure is trivial. Let $H$ be a constant group scheme over $X$ with invertible order. Then the restriction functor
\[
\mathrm{Tor}_{\mathrm{k\'et}}(X,H)
\longrightarrow\;
\mathrm{Tor}_{\mathrm{\'et}}(U,H)
\]
is an equivalence of categories.
\end{corollary}
\begin{example}
    Example \ref{tam} provides an example of an \'etale $\mathbb{Z}/2\mathbb{Z}$-torsor on a dense open subset of $\mathbb{P}^1$ extending to a log \'etale $\mathbb{Z}/2\mathbb{Z}$-torsor over $\mathbb{P}^1_{\mathbb{C}}$.
\end{example}

In positive characteristic, one has more tamely ramified covers, arising from actions of linearly reductive group schemes (for example $\mu_p$ in characteristic $p>0$, cf. Example~\ref{fiber prod}). In this section, we extend the purity theorem to this setting.

\begin{definition}
    Let $X$ be a log scheme. We denote by $\mathbb G_{m,\log}$ the functor
   \begin{align*} (fs/X) & \to (Ab)\\
                      T & \mapsto \Gamma(T,M_T^{gp})
    \end{align*}
    It is a sheaf for the Kummer log flat topology by \cite[Theorem 3.2]{Kato}. 
\end{definition}

\begin{lemma}\label{et=fppf}
Let $X$ be a logarithmic scheme.  Then the natural map
\[
H^1_{\acute{e}t}(X,\mathbf \BG_{m,\log})
\to
H^1_{\mathrm{fppf}}(X,\mathbf \BG_{m,\log})
\]
is an isomorphism. Equivalently, every fppf $\mathbf \BG_{m,\log}$-torsor on
$X_{}$ descends uniquely to an \'etale $\mathbf \BG_{m,\log}$-torsor on
$X_{}$.
\end{lemma}

\begin{proof} 

Let \[\epsilon: (Sch/X)_{fppf} \to (Sch/X)_{\acute{e}t}\]
be the morphism of sites.\\
Given an abelian sheaf $F$ on $(Sch/X)_{fppf}$, the Leray spectral sequence gives
\begin{equation*}
    H^p_{\acute{e}t}(X, R^q\epsilon_*F) \implies H^{p+q}_{fppf}(X,F).
\end{equation*}
It yields an exact sequence
\[0 \to H^1_{\acute{e}t}(X, \epsilon_*F) \to H^1_{fppf}(X,F) \to H^0_{\acute{e}t}(X,R^1\epsilon_*F).\]
We apply this to $F:= \epsilon^*\BG_{m,log}=\BG_{m,log}$. We have that $\BG_{m,log} \to \epsilon_*\epsilon^*\BG_{m,log}$ is an isomorphism. We then get
\[0 \to H^1_{\acute{e}t}(X, \BG_{m,log}) \to H^1_{fppf}(X, \BG_{m,log}) \to H^0_{\acute{e}t}(X,R^1\epsilon_* \BG_{m,log}).\]
It is enough to prove that $R^1\epsilon_*\BG_{m,log}=0$. For that, since it's a sheaf on $(Sch/X)_{\acute{e}t}$, it is enough to show that its geometric stalks vanish. If $X$ is local strictly henselian, then $H^1_{fppf}(X,\BG_{m,log})=H^1_{\acute{e}t}(X,\BG_{m,log})=0$ by Hilbert 90 (\cite[Corollary 3.21]{Niziol2008-K-theory}). This concludes the proof.\\

\end{proof}

\begin{lemma}\label{Gmlog}
    Let $X$ be a log regular log scheme. Let $U$ be the dense open subset of $X$ with trivial log structure. Then the restriction morphism
    \[H^1_{\acute{e}t}(X,\BG_{m,log}) \to H^1_{\acute{e}t}(U,\BG_{m})\]
    is an isomorphism.
\end{lemma}

\begin{proof}
    Write $j: U \hookrightarrow X$ for the open immersion. For $F$ an abelian sheaf on $(Sch/U)_{\acute{e}t}$, the Leray sequence gives
    \[H_{\acute{e}t}^{p}(X,R^q j_*F) \implies H_{\acute{e}t}^{p+q}(U,F) \]
    yielding when applied to $F:=\BG_{m,U}$
    \[0 \to H^1_{\acute{e}t}(X, j_*\BG_{m,U} ) \to H^1_{\acute{e}t}(U,\BG_{m,U}) \to H^0_{\acute{e}t}(X,R^1j_*\BG_{m,U}).\]
By \cite[Theorem 11.6]{KatoToric}, $M_X^{\mathrm{gp}} \simeq j_*\mathcal O_U^\times$. It follows that
 \[0 \to H^1_{\acute{e}t}(X, \BG_{m,log} ) \to H^1_{\acute{e}t}(U,\BG_{m,U}) \to H^0_{\acute{e}t}(X,R^1j_*\BG_{m,U}).\]
   By \cite[Theorem 12.6.38]{Romero}, since $X$ is log regular, we have $R^1j_*\mathcal{O}_U^{\times}=0$. This concludes the proof.
\end{proof}

\begin{lemma}\label{mu_n}
    Let $X$ be a log regular log scheme. Let $U$ be the dense open subset of $X$ with trivial log structure. Then the restriction morphism
    \[H^1_{klf}(X,\mu_n) \to H^1_{fppf}(U,\mu_n)\]
    is an isomorphism.
\end{lemma}

\begin{proof}

 By \cite[Proposition 4.2]{katoLogStructuresII}, we have an exact sequence on $(fs/X)_{klf}$ 
 \[0 \to \mu_n \to \BG_{m,log} \xrightarrow{n} \BG_{m,log} \to 0.\]

 We deduce from the associated log exact sequence a diagram
\[
\begin{tikzcd}[column sep=large, row sep=large]
0 \arrow[r]
& \mathbf \BG_{m,\log}(X)/n \arrow[r] \arrow[d]
& H^1_{\mathrm{klf}}(X,\mu_n) \arrow[r] \arrow[d]
& H^1_{klf}(X,\mathbf \BG_{m,\log})[n] \arrow[r] \arrow[d]
& 0
\\
0 \arrow[r]
& \mathbf \BG_{m}(U)/n \arrow[r]
& H^1_{\mathrm{fppf}}(U,\mu_n) \arrow[r]
& H^1_{fppf}(U,\mathbf \BG_{m})[n] \arrow[r]
& 0
\end{tikzcd}
\]
The left vertical arrow is an isomorphism by \cite[Theorem 11.6]{KatoToric} and the right vertical arrow is an isomorphism because $H^1_{klf}(X,\BG_{m,log})\simeq H^1_{fppf}(X,\BG_{m,log})\simeq H^1_{\acute{e}t}(X,\BG_{m,log})\simeq H^1_{\acute{e}t}(U,\BG_{m}) \simeq H^1_{fppf}(U,\mathbb{G}_m)$, where the first isomorphism follows from \cite[Theorem 3.20]{Niziol}, the second by Lemma \ref{et=fppf} and the third by Lemma \ref{Gmlog}.
 
\end{proof}

We recall the following definition from \cite[Definition 2.4]{AbrOV}:
\begin{definition}[Linearly reductive group]
Let $X$ be a scheme, and let $G$ be an $X$-group scheme. We say that $G$ is linearly reductive if the functor
\begin{align*}
    \mathrm{QCoh}^G(X) &\to \mathrm{QCoh}(X)\\
    F &\mapsto F^G
\end{align*}
is exact.
\end{definition}

\begin{remark}
When $X=\Spec k$ for a field $k$, the category $ \mathrm{QCoh}(X)$ of quasi-coherent sheaves on $X$ is equivalent to the category of $k$-vector spaces, while the category $\mathrm{QCoh}^G(X)$ of quasi-coherent sheaves on $X$ endowed with an action of $G$ is equivalent to the category of $k$-representations of $G$. The definition above is therefore equivalent to the classical notion stating that every finite-dimensional representation of $G$ is semisimple.
\end{remark}

For completeness, we recall the following description of torsors due to Olsson (cf. \cite{Olsson}). Let $X$ be a local strictly Henselian scheme with a log structure and let $G$ be a finite flat linearly reductive group scheme over $X$. Let $Y \to X$ be a log $G$-torsor. By \cite[Lemma 2.20]{AbrOV}, there exists a group extension
\begin{equation*}
0 \longrightarrow \Delta \longrightarrow G \xrightarrow{p} H \longrightarrow 0
\end{equation*}
where $\Delta$ is a diagonalizable group and $H$ a constant group with invertible order. Therefore, from the log $G$-torsor $Y \to X$, we get the data:
\begin{itemize}
    \item A log $\Delta$-torsor $Y \to Y/\Delta$,
    \item a log $H$-torsor $Y/\Delta \to X$,
    \item a morphism of log $\Delta$-torsors $\chi_Y: G \times^{\Delta} Y \longrightarrow H \times_X Y, (g,y) \mapsto (p(g),gy)$ over $H \times_X (Y/\Delta)$.
\end{itemize}
Given $G$ as above, Olsson then defines the category:

\noindent \textbf{Definition 2.5.} Let $\mathcal{C}$ be the category whose objects are triples $(T \to X, Y \to T, \chi_Y)$ where:
\begin{enumerate}
    \item $T \to X$ is a log $H$-torsor.
    \item $Y \to T$ is a log  $\Delta$-torsor.
    \item $\chi_Y : G \times^{\Delta} Y \to H \times_X Y$ is a morphism of log $\Delta$-torsors over $H \times_X T$.
    \item The diagram
    \[
\begin{array}{ccc}
(G \times^{\Delta} G) \times^{\Delta} Y 
& \xrightarrow{m \times \mathrm{id}} 
& G \times^{\Delta} Y \\
\| && \| \\
G \times^{\Delta} (G \times^{\Delta} Y) 
& \xrightarrow{\mathrm{id} \times \chi_Y} 
& G \times^{\Delta} Y
\end{array}
\]
    commutes, where $m$ is the multiplication in $G$.
\end{enumerate}

\bigskip
\begin{theorem}[\cite{Olsson}] \label{Ols}
\begin{align*}
\mathrm{Tor}_{\mathrm{klf}}(X, G) &\longrightarrow \mathcal{C} \\
Y &\longmapsto (Y/\Delta \to X, \, Y \to Y/\Delta, \, \chi_Y)
\end{align*}
is an equivalence of categories.
\end{theorem}

We can now prove our main result of this section:

\begin{theorem}\label{lin-red}
    Let $X$ be a log regular log scheme and let $U$ be the dense open subset of $X$ with trivial log structure. Let $G$ be a finite flat linearly reductive $X$-group scheme. Then the restriction functor
\[
\mathrm{Tor}_{\mathrm{klf}}(X,G)
\;\longrightarrow\;
\mathrm{Tor}_{\mathrm{fppf}}(U,G)
\]
    is an equivalence of categories.
\end{theorem}

\begin{proof} 
In \cite[Proposition 4.3]{GillibertTame}, it was proved that when $G$ is finite flat over $X$, the functor 
\[
\mathrm{Tor}_{\mathrm{klf}}(X,G)
\;\longrightarrow\;
\mathrm{Tor}_{\mathrm{fppf}}(U,G).
\]
is fully faithful when $X$ is regular. The proof adapts straightforwardly to the log regular case, which we briefly summarize here. By standard descent arguments, the proof reduces to showing that the restriction map $G(X)\longrightarrow G(U)$ is bijective. Injectivity follows from the fact that $G$ is separated. On the other hand, since $X$ is log regular, it is normal by \cite[Theorem 4.2]{KatoToric}, hence any section $U \to G$ extends to a section $X \to G$ by \cite[Corollary 6.1.14]{Grothendieck1961EGAII}.

Essential surjectivity relies on Lemma \ref{mu_n} and the description in Theorem \ref{Ols}. By \'etale descent, we may assume that $X$ is local and strictly henselian. An fppf $G$-torsor $Z \to U$ decomposes as an fppf $\Delta$-torsor $Z \to Z/\Delta$ and an \'etale $H$-torsor $Z/\Delta \to U$. To extend the fppf $G$-torsor $Z \to U$ over $X$, we show that it suffices to extend each of the two torsors in this decomposition.
Since $X$ is log regular, the \'etale $H$-torsor $Z/\Delta \to U$ extends uniquely to a Kummer log \'etale $H$-torsor $T \to X$ by Corollary \ref{inv}. Then $T \to X$ is log \'etale by klf descent and hence $T$ is log regular by \cite[Theorem 8.2]{KatoToric}. By Lemma \ref{mu_n}, the $\Delta$-torsor $Z \to Z/\Delta$ extends uniquely to a klf $\Delta$-torsor $Y \to T$. By Theorem \ref{Ols}, it is left to show that the morphism $\chi_Z$ also extends. The restriction map
\[\mathrm{Tor}_{kfl}(H \times T, \Delta) \to \mathrm{Tor}_{fppf}(H_U \times T_U, \Delta)\]
is fully faithful as argued previously given that $\Delta$ is finite flat and $H \times_X T$ is log regular (it is even an equivalence of categories).
\end{proof}

\subsection{Moduli of torsors on nodal curves}
Let $S$ be a log regular log scheme and let $X/S$ be a log curve as in 
\cite[Definition~2.4.1.1]{molcho_wise_2022}. In particular, $X \to S$ is log smooth. 
Let $U \subset S$ be the dense open subset where the log structure is trivial, so that 
$X_U \to U$ is smooth.

In \cite{saraT}, the authors studied the problem of extending $G$-fppf torsors from $X_U$ to $X$, and related this question to extending some group homomorphisms into the \textit{logarithmic Jacobian} of $X/S$, as defined in \cite{molcho_wise_2022} and \cite{HMOP}. 
More precisely, they prove the following.

\begin{theorem}\cite[Corollary 3.14]{saraT}
Let $X/S$ be a log curve and assume that $S$ is log regular with $U \subset S$ the dense open with trivial log structure. Let $G$ be a finite flat group scheme over $S$ such that its Cartier dual $G^D$ is \'etale. 
Then there is an isomorphism
\[
H^1_{\mathrm{klf}}(X,G) \big/ H^1_{\mathrm{klf}}(S,G)
\;\simeq\;
H^1_{\mathrm{fppf}}(X_U,G) \big/ H^1_{\mathrm{fppf}}(U,G).
\]
\end{theorem}

Group schemes satisfying the hypothesis of the theorem are essentially diagonalizable 
group schemes. In what follows, we extend this result to linearly reductive group schemes.

\begin{corollary}\label{logcurve}
Let $X/S$ be a log curve and assume that $S$ is log regular with $U \subset S$ the dense open with trivial log structure. Let $G$ be a finite flat linearly reductive group scheme over $S$. Then the restriction morphism
\[
H^1_{\mathrm{klf}}(X,G) \big/ H^1_{\mathrm{klf}}(S,G)
\;\longrightarrow\;
H^1_{\mathrm{fppf}}(X_U,G) \big/ H^1_{\mathrm{fppf}}(U,G)
\]
is an isomorphism.
\end{corollary}

\begin{proof}
Since $X \to S$ is log smooth and $S$ is log regular, it follows that $X$ is log regular by \cite[Theorem 8.2]{KatoToric}. The result therefore follows from Theorem~\ref{lin-red}.
\end{proof}

\section{Applications}
\subsection{Torsors on root stacks}
In this section, we establish a connection between the compactification of torsors in the logarithmic setting seen in the previous section and the valuative criterion of properness for algebraic stacks \cite{root,rootbis}. To this end, we first establish an equivalence between logarithmic torsors and torsors over root stacks. The relationship between logarithmic geometry and root stacks has been studied in \cite{BorneVis,TalpoV,Olssonn}. We briefly recall the construction of root stacks and infinite root stacks of log schemes.

There is an equivalent way to define log structures using symmetric monoidal functors, which naturally leads to the stack-theoretic perspective. Namely, root stacks are constructed by modifying the logarithmic structure. We refer the reader to \cite{BorneVis} for background on symmetric monoidal categories and functors.

\begin{definition}[Deligne--Faltings structure]
A fine and saturated Deligne--Faltings structure (DF structure) on a scheme ${X}$ consists of a symmetric monoidal functor
\[
L : A \longrightarrow \mathrm{Div}_{X_{\'et}}
\]
with trivial kernel, where $A$ is a sheaf of monoids on the small \'etale site of ${X}$ whose stalks are fine and saturated. The condition of trivial kernel means that for every \'etale morphism ${U} \to {X}$, the only section of $A({U})$ whose image is isomorphic to $(\mathcal{O}_U,1)$ is the zero element.
\end{definition}

A DF structure determines a log structure on ${X}$. Indeed, one sets \[\mathcal{M}_X := A \times_{\mathrm{Div}_X} \mathcal{O}_X\] where the morphism $\mathcal{O}_X \longrightarrow [\mathcal{O}_X/\mathcal{O}_X^\times] = \mathrm{Div}_{X_{\'et}}$ is the natural quotient map. Conversely, from a log structure $\mathcal{M}_X$, one can define a DF structure by setting $A:=\overline{\mathcal{M}}_X$. We refer to \cite{BorneVis} for the equivalence of the two definitions.

\begin{definition}[Root stack of a log scheme]
Let $X$ be a logarithmic scheme. The root stack $\sqrt[n]{X}$ is the stack over schemes whose objects over a morphism $T \to \underline{X}$ consist of extensions
\[
\begin{tikzcd}
\overline{\mathcal{M}}_{X|T} \arrow[r] \arrow[d]
& \mathrm{Div}_{T_{\'et}} \\
\tfrac1n \overline{\mathcal{M}}_{X|T} \arrow[ur,dashed] &
\end{tikzcd}
\]
of the pullback DF structure $\overline{\mathcal{M}}_{X|T} \to \mathrm{Div}_{T_{\'et}}$ from $X$ to $T$ to the sheaf of monoids $\tfrac1n \overline{\mathcal{M}}_{X|T}$.

\end{definition}

Root stacks form an inverse system with respect to the natural order given by inclusion, which allows one to define:
\begin{definition}[Infinite root stack]\cite[Definition 3.3 and Proposition 3.5]{TalpoV}
The infinite root stack of $X$ is defined to be
\[
\sqrt[\infty]{X}=\varprojlim_n \sqrt[n]{X}.
\]
\end{definition}

\begin{remark}\label{niel}
 $\sqrt[n]{X}$ is a tame algebraic stack with coarse moduli space $X$ by \cite[Proposition 4.19]{BorneVis}. However, $\sqrt[\infty]{X}$ is not an algebraic stack (its diagonal is not of finite type). Nevertheless, it admits an fpqc atlas and, \'etale locally on its coarse moduli space $X$, a presentation as a quotient stack \cite[Corollary 3.13]{TalpoV}.
\end{remark}

\begin{remark}\label{rs}
When the logarithmic structure of $X$ is induced by an irreducible effective Cartier divisor $D \subseteq X$, so that $\mathcal{\overline{M}}_X$ is the constant sheaf ${\mathbb{N}}$ on $D$, the root stack $\sqrt[n]{X}$ defined above coincides with the usual root stack $\sqrt[n]{(\underline{X},D)}$ introduced in \cite{Cadman}.
\end{remark}

\begin{example}
Let $R$ be a discrete valuation ring with uniformizing parameter $\pi$ and residue field $k := R/(\pi)$. For a positive integer $n$, we denote by $\sqrt[n]{\Spec R}$ the $n$-th root stack taken with respect to the Cartier divisor $\Spec k \subset \Spec R$. It can be described as the quotient stack $\left[ \Spec R[t]/(t^n - \pi) \, / \, \mu_n \right]$, where $\mu_n$ acts on $\Spec R[t]/(t^n - \pi)$ by multiplication on $t$. On the other hand, if we endow $\Spec R$ with the divisorial log structure induced by the closed point, then the characteristic monoid is $\overline{\mathcal{M}}_{\Spec R} \simeq \mathbb{N}$ over the closed point, and the root stack of the induced log scheme is constructed via the Kummer map $\mathbb{N} \to \tfrac{1}{n}\mathbb{N}$. It is, by \cite[Corollary 3.13]{TalpoV}, isomorphic to the quotient stack $[\Spec R \times_{\Spec \mathbb{Z}[\mathbb{N}]} \Spec \mathbb{Z}[\frac{1}{n}\mathbb{N}]/\mu_n]$. This verified Remark \ref{rs} in this case.
\end{example}

\begin{theorem}\label{equi}
Let $X$ be a log scheme, and let $G$ be a finite flat group scheme over $X$. There is an equivalence of categories
\[
\mathrm{Tor}_{klf}(X,G) \equiv \mathrm{Tor}_{fppf}(\sqrt[\infty]{X},G_{})
\]
where the category on the right denotes that of fppf $G_{}$-torsors on $\sqrt[\infty]{X}$ which are representable morphisms.
\end{theorem}

\begin{proof}
By \cite[Theorem 6.16]{TalpoV}, there is a continuous functor of sites $F: X_{klf} \to \sqrt[\infty]{X}_{fppf}$ that induces an equivalence of their associated topoi:
\[
F_* : \mathrm{Sh}(X_{klf}) \xrightarrow{\sim} \mathrm{Sh}(\sqrt[\infty]{X}_{fppf}).
\]
Because an equivalence of topoi preserves finite limits, group objects, and group actions, it naturally induces an equivalence between the categories of sheaf torsors:
\[
\Phi: \mathrm{Sheaf~Torsors}(X_{klf}, G) \xrightarrow{\sim} \mathrm{Sheaf~Torsors}(\sqrt[\infty]{X}_{fppf}, G_{\infty}),
\]
where $G_{\infty}:=F_*G=G \times_X \sqrt[\infty]{X}$.\\
On the left-hand side, since $G$ is a finite flat group scheme, it follows from \cite[Theorem 2.1]{GillibertTame} that any $G$-torsor in the category of sheaves $\mathrm{Sh}(X_{klf})$ is representable by a scheme. Thus, the category $\mathrm{Sheaf~Torsors}(X_{klf}, G)$ is naturally equivalent to the category of representable log $G$-torsors $\mathrm{Tor}_{klf}(X,G)$.

On the right-hand side, let $\mathcal{P} \in \mathrm{Sheaf~Torsors}(\sqrt[\infty]{X}_{fppf}, G_{\infty})$. We check that the structural morphism $\mathcal{P} \to \sqrt[\infty]{X}$ is a representable morphism of stacks. Let $U$ be a scheme mapping to $\sqrt[\infty]{X}$ via a morphism $U \to \sqrt[\infty]{X}$. The fiber product (2-pullback) $\mathcal{P} \times_{\sqrt[\infty]{X}} U \to U$ is an fppf $G$-torsor over the scheme $U$. By applying \cite[Theorem 2.1]{GillibertTame} again, this pullback $\mathcal{P} \times_{\sqrt[\infty]{X}} U$ is representable by a scheme. Hence the map $\mathcal{P} \to \sqrt[\infty]{X}$ is representable. The functor $\Phi$ therefore restricts to the desired equivalence of categories:
\[
\mathrm{Tor}_{klf}(X,G) \equiv \mathrm{Tor}_{fppf}(\sqrt[\infty]{X},G_{}).
\]
Furthermore, under the equivalence $\Phi$, if $Y \to X$ is a log $G$-torsor, its image $\Phi(Y)$ is given by the induced morphism $\sqrt[\infty]{Y} \to \sqrt[\infty]{X}$ in \cite[Theorem 6.16]{TalpoV}. 
\end{proof}

\begin{prop}
Under the assumptions of Theorem \ref{equi} and if the underlying scheme of $X$ is quasi-compact and quasi-separated, an fppf $G_{}$-torsor on $\sqrt[\infty]{X}$ factors through $\sqrt[n]{X}$ for some positive integer $n$.
\end{prop}
\begin{proof} 
Since $G$ is finite flat, the classifying stack $BG$ is smooth, hence locally of finite presentation by \cite[Tag 0DLS]{stacks-project}. By Remark \ref{niel} and \cite[Theorem 1.1]{Keel}, the morphism $\sqrt[n]{X} \to X$ is proper, hence $\sqrt[n]{X}$ is quasi-compact and quasi-separated. By \cite[Proposition B1]{Ryd}, the natural functor
\[
\varinjlim_{i} \mathrm{Mor}(\sqrt[i]{X}, BG) \longrightarrow \mathrm{Mor}(\sqrt[\infty]{X}, BG)
\]
  is an equivalence of categories.
\end{proof}

\begin{remark}
 The first inspiration for this comparison theorem comes from \cite[Proposition 3.5]{BB}, where the authors introduce a notion of tamely ramified torsors, which can be viewed as the local version of Kummer log flat torsors. They establish an equivalence between such torsors and torsors on some (classical) root stacks when the structural group $G$ is abelian and the divisor is a strict normal crossing. 
\end{remark}
\begin{remark}
   Another instance of a comparison between Kummer log flat torsors and fppf torsors over tame stacks appears in \cite{GillibertTame}. The authors prove that, given a \(G\)-log torsor \(Y \to X\), the quotient stack \([Y/G]\) is a tame stack in the sense of \cite{AbrOV}. They subsequently introduce a notion of \emph{tame cover} and show that this construction defines a fully faithful functor from the category of \(G\)-log torsors to the category of tame covers \cite[\S 4]{GillibertTame}. Their proof assumes that \(X\) is regular and endowed with the logarithmic structure associated with a normal crossing divisor.
\end{remark}

\subsection{Valuative criterion of properness for proper tame stacks}
We recall the statement of the criterion:

\begin{theorem}(\cite[Theorem 3.1]{root})
Let $f:\mathfrak{X}\to\mathfrak{Y}$ be a proper tame morphism of algebraic stacks, $R$ a discrete valuation ring with fraction field $K$, and suppose given a 2-commutative diagram
\[
\begin{tikzcd}
\Spec K \arrow[r] \arrow[d] & \mathfrak{X} \arrow[d,"f"]\\
\Spec R \arrow[r] & \mathfrak{Y}.
\end{tikzcd}
\]
Then there exists a unique positive integer $n$ and a representable morphism $\sqrt[n]{\Spec R}\longrightarrow \mathfrak{X}$ lifting the given morphism $\Spec K\to\mathfrak{X}$, such that the diagram
\[
\begin{tikzcd}
&\Spec K \arrow[r] \arrow[d] &\mathfrak{X} \arrow[d,"f"]\\
\sqrt[n]{\Spec R} \arrow[urr] \arrow[r]&\Spec R\arrow[r] &\mathfrak{Y}
\end{tikzcd}
\]
is 2-commutative. 
\end{theorem}

\begin{remark}
    When $G$ is finite flat and linearly reductive, the classifying stack $BG$ is tame by \cite[\S 3]{AbrOV}. Furthermore, it is proper. Indeed, consider the Cartesian diagram where the bottom horizontal map is the diagonal:
\[
\begin{tikzcd}
G \arrow[r] \arrow[d] \arrow[dr, phantom, "\lrcorner", very near start] & X \arrow[d] \\
BG \arrow[r,"\Delta"] & BG \times_X BG
\end{tikzcd}
\]
By base change and fppf descent, $G$ is proper if and only if $BG$ is separated. Since $BG$ is smooth (\cite[Tag 0DLS]{stacks-project}), it is locally of finite type. Moreover, quasi-compactness follows from the existence of a surjective morphism $X \to BG$. Hence $BG$ is of finite type. 

\end{remark}

The valuative criterion above shows that morphisms from the generic point of a discrete valuation ring into a proper tame stack extend after passing to a suitable root stack. The following corollary may be viewed as a higher-dimensional analogue of this phenomenon in the special case of the classifying stack $BG$.

\begin{corollary}\label{cor}
Let $X$ be either 
\begin{enumerate}
    \item a toric variety with dense torus $U$ and toric divisor $D$,
    \item a nodal proper curve with $U$ the smooth locus and $D:=X\backslash U$.
\end{enumerate}
Let $G$ be a finite flat linearly reductive group scheme over $X$. Then any rational map \begin{tikzcd}[column sep=0.8em]
X \arrow[dashed]{r} & BG
\end{tikzcd} extends (in a unique way) to a morphism $\sqrt[n]{X} \to BG$ for some integer $n$. In addition, if $X$ is regular and $D$ is a strict normal crossing divisor, then for every point $x \in D$, the induced fppf $G$-torsor on $U$ specializes to an fppf $G$-torsor over $\Spec k(x)$.
\end{corollary}

\begin{proof}
In the first case, $X$ is endowed with the divisorial log structure associated to the toric boundary. In the second case, $X$ is equipped with the logarithmic structure making it log smooth over a point (cf. \cite{Kat}). In both cases, $X$ is log regular. By Theorem \ref{lin-red}, any fppf $G$-torsor over the dense open of triviality $U$ extends uniquely to a log $G$-torsor over $X$. We then use Theorem \ref{equi} to conclude the first statement because $X$ is quasi-compact and separated. For the second statement, we note that a morphism $\Spec k(x) \to X$ lifts to $\sqrt[n]{X}$ if and only if the induced morphism $\overline{\mathcal{M}}_{X|\Spec k(x)} \to k(x)$ lifts to a morphism $\frac{1}{n}\overline{\mathcal{M}}_{X|\Spec k(x)} \to k(x)$. Note that $\overline{\mathcal{M}}_{X|\Spec k(x)} \simeq \mathbb{N}^r$ and the map defining the log structure at $x$ is given by $\mathbb{N}^r \oplus k(x)^\times \to k(x)$, $(a,u)\mapsto 0$ for $a\neq 0$ and $(0,u)\mapsto u$.
\end{proof}
\subsection{Log Nori fundamental group}
The construction below of the logarithmic \'etale fundamental group appears for example in \cite[Section 4]{illusie2002overview} and \cite[Section 10]{katoLogStructuresII} .\\

Let $X$ be a log scheme with a connected underlying scheme. Let $\overline{x} \to X$ be a log geometric point of $X$ as defined in \cite[Definition 4.1]{illusie2002overview}. We denote %by $(\mathrm{kl\acute{e}}(X))$ the category consisting of finite Kummer log \'etale covers of $X$and 
by $(\mathrm{fsets})$ the category of finite sets. The fiber functor
\[
F : \mathrm{fk\acute{e}t}(X) \longrightarrow (\mathrm{fsets}), 
\qquad T \longmapsto F(T) := T_{\overline{x}},
\]
is pro-representable by a pro-object $\widetilde{X}$ of 
$\mathrm{fk\acute{e}t}(X)$, and if $\pi_1^{\log,\acute{e}t}(X,\overline{x})$ denotes the profinite group $\mathrm{Aut}(F)$, then $F$ induces an equivalence of categories
\[
\mathrm{fk\acute{e}t}(X)
\;\xrightarrow{\sim}\; {\pi_1^{\log, \acute{e}t}(X,\overline{x})}-(\mathrm{fsets}),
\]
where the right-hand side denotes the category of finite sets endowed with a 
continuous action of $\pi_1^{\log, \acute{e}t}(X,\overline{x})$. The group $\pi_1^{\log, \acute{e}t}(X,\overline{x})$ is called the \emph{logarithmic \'etale fundamental group} 
of $X$ at $\overline{x}$, and $\widetilde{X}$ is called a \emph{log universal cover} of $X$. 
If $G$ is a finite (abstract) group, Kummer \'etale Galois covers of $X$ with group $G$ (i.e.\ Kummer \'etale covers $Y \to X$ endowed with an action of $G$ by 
$X$-automorphisms such that $Y$ is a Kummer log \'etale $G$-torsor over $X$) are pushed-out from $\widetilde{X}$ along continuous homomorphisms $\pi_1^{\log, \acute{e}t}(X,\overline{x}) \longrightarrow G$.\\

 Now, we fix an algebraically closed field $k$ of characteristic $0$. 
Then every $k$-group scheme is \'etale, and hence every fppf 
$G$-torsor is \'etale by descent. In particular, any construction of 
an analogue of the Nori fundamental group of a $k$-log scheme $X$ in this setting recovers the logarithmic \'etale fundamental group recalled above. Next, our goal is to construct an analogue of the Nori fundamental group which classifies Kummer log flat torsors in positive characteristic.

%\subsection{Log Nori fundamental group}
We fix for the rest of the section a field $k$ of characteristic $p>0$. Let $X$ be a $k$-log scheme and let $x_0$ be a $k$-log point of $X$ (i.e. a section $\Spec k \to X$ such that the underlying scheme $\Spec k$ of $x_0$ is endowed with a log structure making $x_0 \to X$ a morphism of log schemes) . Consider the category $\mathsf{FT}(X)_{x_0}$ whose objects are triples $(T,G,t_0)$ where
\begin{itemize}
 \item $G$ is a finite linearly reductive group scheme over $k$,
  \item $T \to X$ is a Kummer log flat $G$-torsor,
   \item $t_0$ is a $k$-log point of $T$ lying over $x_0$ (we say that $T$ is pointed over $x_0$).
\end{itemize}
A morphism $(f,g)\colon (T,G,t_0)\to (T',G',t_0')$ is a morphism of log torsors
sending $t_0$ to $t_0'$. We denote by $\mathsf{PT}(X)_{x_0}$ the category of triples as above where we allow $G$ to be a profinite group scheme.

\begin{definition}[Log Nori Fundamental Group]
A profinite group scheme $\pi_1^{\mathrm{N},log}(X,x_0)$ is called a \emph{Log Nori Fundamental Group} of $X$ if there exists a triple 
\[
(\widetilde{T}, \pi_1^{\mathrm{N}, log}(X,x_0), \widetilde{t}_0)\in \mathsf{PT}(X)_{x_0}
\]
such that for every object $(T,G,t_0)$ in $\mathsf{FT}(X)_{x_0}$ there exists a unique morphism
\[
(\widetilde{T}, \pi_1^{\mathrm{N}, log}(X,x_0), \widetilde{t}_0)\longrightarrow (T,G,t_0).
\]
We call $\widetilde{T}$ the \emph{universal log torsor} over $X$. %In particular, $(\widetilde{T}, \pi_1^{N,log}(X,x_0), \widetilde{t}_0)$ is initial in $PT(X)_{x_0}$.
\end{definition}
In what follows, we construct this fundamental group when $X$ is log regular. 
 
\begin{lemma}\label{prod} Assume that $X$ is log regular with integral underlying scheme and let $U$ be the dense open subset where the log structure is trivial. Consider a pair of morphisms
\[
(f_i,\pi_i)\colon (T_i,G_i,t_i) \longrightarrow (T_0,G_0,t_0),
\qquad i=1,2,
\]
of  klf torsors over $X$ pointed over $x_0$. Then the triple
\[
\bigl(T_1\times_{T_0} T_2,\; G_1\times_{G_0} G_2, (t_1, t_2)\bigr)
=:\bigl(T^{\times},G^{\times}, t^{\times}\bigr)
\]
is a klf torsor over $X$, pointed over $x_0$. 
\end{lemma}

\begin{proof} 
We proceed in several steps.
\begin{itemize}
    \item \textbf{Step 1: $T^{\times} \to X$ is surjective on sets}. For $i \in \{1,2\}$, if $p_i: T_1 \times_X T_2 \to T_i \to T_0$, there exists a (unique) morphism $g:T_1 \times_X T_2 \to G_0$ such that $gp_1=p_2$ (because $T_0 \to X$ is a klf $G_0$-torsor). In particular, considering the natural map $T^{\times} \to T_1 \times_X T_2$ and if $\epsilon: \Spec k \to G_0$ is the (strict) identity section, we have that $T^{\times}$ is the closed subscheme $g^{-1}(\epsilon) \subset T_1 \times_X T_2$. Now, we know that $\pi: T_1 \times_X T_2 \to X$ is a klf $(G_1 \times G_2)$-torsor. By \cite[Theorem 2.1]{GillibertTame}, $\pi$ has a finite underlying morphism of schemes, and since $X$ is locally noetherian, $\pi$ is locally of finite presentation. It follows from \cite[Proposition 2.5]{katoLogStructuresII} that $\pi$ is open. But it is also closed as it is finite. On the other hand, since $G_0$ is finite over $k$, the connected components of $G_0$ have a single element each, hence $\epsilon= G^0_0$ as sets, where $G^0_0$ is the identity component of $G_0$. Since $G^0_0$ is open and closed, $g^{-1}(G^0_0)$ is open and closed, and therefore, $\pi(\underline{T}^{\times})=\pi(g^{-1}(\epsilon))=\pi(g^{-1}(G^0_0))$ (as sets) is open and closed in $\underline{X}$. But since $\underline{X}$ is connected, $\pi(\underline{T}^{\times})=\underline{X}$ as sets, hence the claim.\\

    \item \textbf{Step 2: the restriction $T_U^{\times} \to U$ is an fppf $G^{\times}$-torsor.} Since $T^{\times} \to X$ is surjective on sets, $T^{\times}_U$ (where $T^{\times}_U=T_{1,U} \times_{T_{0,U}} T_{2,U}= T^{\times} \times_X U$) is not empty. Since the action of $G^{\times}$ on $T^{\times}_U$ is free and $T^{\times}_U \to U$ is affine, it follows that there exists a closed subscheme $Z \subset U$ such that $T^{\times}_U \to Z$ is an fppf $G^{\times}$-torsor. Since $U$ is reduced and connected, $Z=U$ by \cite[Chapter II]{Nori}. Therefore,
  \[T^{\times}_U \simeq T_{1,U} \times_{T_{0,U}}T_{2,U} \to U\]
  is an fppf $G^{\times}$-torsor. By Lemma \ref{linred}, $G^{\times}$ is linearly reductive. Hence, by Theorem \ref{lin-red}, $T^{\times}_U $ extends uniquely to
  \[\widetilde{T} \to X\]
  a klf $G^{\times}$-torsor (note that if $G \to \Spec k$ is finite linearly reductive, then $G_X \to X$ is finite flat linearly reductive by \cite[Proposition 2.6]{AbrOV}).\\

  \item \textbf{Step 3: $\widetilde{T}\simeq T^{\times}$ as $G^{\times}$-log schemes over $X$.}

  Denote the projections \[\rho_i: G^{\times}=G_1 \times_{G_0} G_2 \to G_i,~~~~~~i \in \{0,1,2\}.\]
  Let $\widetilde{T} \times^{G^{\times}}G_i \to X$ be the contracted product corresponding to the klf $G_i$-torsor obtained by push-out of $\widetilde{T} \to X$ along $\rho_i$. \\
  \underline{Claim 1}: $\widetilde{T} \times^{G^{\times}}G_i \simeq T_{i}$ as klf $G_i$-torsors over $X$.\\
  We consider the projections
\[\psi_{i,U}: \widetilde{T}_U=T^{\times}_U=T_{1,U} \times_{T_{0,U}} T_{2,U} \to T_{i,U}\]
which are equivariant with respect to $\rho_i$. We now construct
\begin{align*}
    \Phi_{i,U}:  \widetilde{T}_U \times G_{i} = (T_{1,U} \times_{T_{0,U}} T_{2,U} ) \times G_{i} &\to  T_{i,U}\\
     (t,g) &\mapsto \psi_{i,U}(t).g
\end{align*}
Let $h \in G^{\times}$. Consider the action of $G^{\times}$ on  $\widetilde{T}_U \times G_{i}$  given by $(t,g).h= (t.h, \rho_i(h)^{-1}g)$. We have that
\[\Phi_{i,U}((t,g).h)=\Phi_{i,U}(t.h, \rho_i(h)^{-1}g)=\psi_{i,U}(t.h)\rho_i(h)^{-1}g=\psi_{i,U}(t)\rho_i(h)\rho_i(h)^{-1}g=\Phi_{i,U}(t,g)\]
where the next-to-last equality follows from the fact that $\psi_{i,U}$ is $\rho_i$-equivariant. Therefore, we get a map
\[\overline{\Phi}_{i,U}:\widetilde{T}_U \times^{G^{\times}} G_{i} \to T_{i,U}\]
which is $G_i$-equivariant. Since both sides are fppf $G_i$-torsors over $U$, it is an isomorphism of $G_i$-torsors. By Theorem \ref{lin-red} (uniqueness of extension), we deduce an isomorphism of klf $G_i$-torsors
\[\widetilde{T} \times^{G^{\times}} G_i \xrightarrow[h_i]{\simeq} T_i,\]
which proves the claim.\\
\underline{Claim 2:} there exists an isomorphism $\widetilde{T} \simeq T_1 \times_{T_0} T_2$ over $X$ which is $G^{\times}$-equivariant.\\
Note that for $i \in \{1,2\}$, the composition \[\widetilde{T}_U \times^{G^{\times}} G_i  \xrightarrow[h_i]{\simeq} T_{i,U} \xrightarrow{f_i} T_{0,U} \xrightarrow[h_0^{-1}]{\simeq} \widetilde{T}_U \times^{G^{\times}} G_0 \] is given by
\begin{align*}
    \widetilde{T}_U \times^{G^{\times}} G_i & \to \widetilde{T}_U \times^{G^{\times}} G_0,\\
    (t,g_i) &\mapsto (t,\pi_i(g_i)).
\end{align*}
We would like to prove that the extended $\pi_i$-morphism
\[\widetilde{T} \times^{G^{\times}} G_i \xrightarrow[h_i]{\simeq} T_i \xrightarrow{f_i} T_0 \xrightarrow[h_0^{-1}]{\simeq} \widetilde{T} \times^{G^{\times}} G_0\]
is given by $(t,g_i) \mapsto (t,\pi_i(g_i))$ as well. For this, it is enough to show that if $H_1$ and $H_2$ are finite flat $X$-group schemes with a map $H_1 \to H_2$, and if $Y_1 \to X$ is a klf $H_1$-torsor and $Y_2 \to X$ a klf $H_2$-torsor, then the restriction map of morphisms of torsors 
\[\mathrm{Hom}_{X}(Y_1,Y_2) \to \mathrm{Hom}_{U}(Y_{1,U},Y_{2,U})\]
is injective. Let $f,g: Y_1 \to Y_2$ be two morphisms of torsors over $X$ such that $f_U=g_U$. Consider the Cartesian diagram

\[
\begin{matrix}
E & \longrightarrow & Y_2 \\
\downarrow & & \downarrow \\
Y_1 & \longrightarrow & Y_2 \times_X Y_2
\end{matrix}
\]
where the bottom horizontal map is the map $(f,g)$ and the right vertical map is the diagonal $\Delta_{2}$ (we take the fs fiber product). The diagonal $\Delta_2$ is closed. Indeed, if we compose it with the isomorphism $Y_2 \times_X Y_2 \simeq H_2 \times Y_2$ (it is an isomorphism because $Y_2 \to X$ is a klf $H_2$-torsor), we get the morphism induced by the section $X \to H_2$, which is a closed immersion since $H_2 \to X$ is separated (since it is finite). Therefore, $E \to Y_1$ is a closed immersion. But it contains a dense open subset of $Y_1$, hence $E=Y_1$ and therefore $f=g$. This concludes the proof of the claim. \\

Now, it remains to notice that given how the maps $\widetilde{T} \times^{G^{\times}} G_i \to \widetilde{T} \times^{G^{\times}} G_0$ are defined, we have
\[T_1 \times_{T_0} T_2 \simeq \widetilde{T} \times^{G^{\times}}G_1 \times_{\widetilde{T} \times^{G^{\times}}G_0}\widetilde{T} \times^{G^{\times}}G_2 \simeq \widetilde{T} \times^{G^{\times}} G_1 \times_{G_0} G_2 \simeq \widetilde{T}.\]
This finishes the proof of the lemma.
\end{itemize}
\end{proof}
\begin{lemma}\label{linred}
    Let $G_1$, $G_2$ and $G$ be finite linearly reductive $k$-group schemes. Assume that we have group homomorphisms $G_1 \to G_0$ and $G_2 \to G_0$. Then $G_1 \times_G G_2$ is a finite linearly reductive $k$-group scheme.
\end{lemma}
\begin{proof}
    $G_1 \times_k G_2$ is finite linearly reductive by \cite[Proposition 2.6]{AbrOV}. Furthermore, if $\phi_i: G_i \to G_0$, then $G_1 \times_{G_0} G_2$ is the equalizer of $\phi_i \circ p_i: G_1 \times_k G_2 \to G_0$, hence a (closed) subgroup scheme of $G_1 \times_k G_2$, hence finite and linearly reductive by \cite[Proposition 2.7]{AbrOV}. 
\end{proof}

\begin{prop}
    The category $FT(X)_{x_0}$ is a cofiltered category.
\end{prop}

\begin{proof}
  Indeed, we have that 
\begin{itemize}
\item $FT(X)_{x_0}$ is nonempty: $X \to X$ is a klf torsor.
\item  If $T_1, T_2$ are klf torsors, $T_1 \times_X T_2$ is a klf torsor too, and we have
morphisms $T_1 \times_X T_2 \to T_1, T_1 \times_X T_2 \to T_2$.
\item  Given two morphisms of klf-torsors $f,g: T' \to T$ over $X$, there exists by Lemma \ref{prod} a klf torsor $T' \times_T T'$ over $X$ such that the two compositions $f \circ p_1, g \circ p_2 : T' \times_T T' \to T' \to T$ are equal.
\end{itemize}
\end{proof}

From the previous proposition, we deduce:
\begin{theorem}\label{pi}
Let $k$ be a field of characteristic $p>0$ and let $X$ be a log regular $k$-log scheme whose
underlying scheme is integral. Let $x_0$ be a $k$-log point of $X$. Then there exists a profinite (and pro-linearly reductive) universal object
\[
(\widetilde{T}, \pi_1^{N,\log}(X,x_0), \widetilde{t}_0)
\]
in $PT(X)_{x_0}$ such that for every triple $(T,G,t_0)$ in $FT(X)_{x_0}$, there exists a unique morphism
\[
(\widetilde{T}, \pi_1^{N,\log}(X,x_0), \widetilde{t}_0)
\longrightarrow
(T,G,t_0)
\]
of objects in $PT(X)_{x_0}$. Moreover, there is a bijection between $k$-group scheme morphisms
\[
\pi_1^{N,\log}(X,x_0) \longrightarrow G,
\]
where $G$ is a finite linearly reductive $k$-group scheme, and objects of $FT(X)_{x_0}$. The torsor corresponding to a morphism $\varphi : \pi_1^{N,\log}(X,x_0) \to G$
is the contracted product $\widetilde{T} \times^{\pi_1^{N,\log}(X,x_0)} G$ via $\varphi$ (cofiltered limits commute with finite limits and colimits).
\end{theorem}

\begin{remark}
    If $T \to X$ is a klf torsor under a finite flat group scheme, its underlying scheme is finite by \cite[Theorem 2.1]{GillibertTame}. In particular, if $T_1 \to T_0$ is a morphism of klf torsors under finite flat group schemes, it is affine. Therefore, a projective limit of objects in $FT(X)_{x_0}$ exists as a scheme by \cite[IV3, Proposition 8.2.3]{EGA}. This implies  that $\pi_1^{N,log}(X,x_0)$ and $\widetilde{T}$ have underlying structures of schemes.
\end{remark}

\begin{corollary}\label{U}Assume that $X$ is a log regular log scheme with an integral underlying scheme and let $U$ denote the dense open of triviality. Let $x_0$ be a $k$-log point of $X$ with trivial log structure (i.e. $x_0 \in U$). We have an isomorphism 
\[\pi_1^{N,log}(X,x_0) \simeq \pi_1^N(U,x_0)\]
where $\pi_1^N$ is the maximal pro-linearly reductive Nori fundamental group on $U$.
\end{corollary}

\begin{proof} Note that if $x_0$ is endowed with the trivial log structure, then $x_0 \to X$ factors through $x_0 \to U$. The corollary follows from the equivalence of categories proved in Theorem \ref{lin-red}.
\end{proof}

\begin{remark}
    Similarly to the classical setting, $\pi_1^{N,log}(X,x_0) $ is a group scheme-theoretic analogue of the maximal prime-to-$p$ quotient of the log \'etale fundamental group. In particular, it follows from \cite[Example 4.7(c)]{illusie2002overview} that if $X$ is a regular integral scheme endowed with a log structure induced by a normal crossing divisor $D$, $U:=X \backslash D$ and $\widetilde{x_0}$ a log geometric point over a $k$-log point $x_0 \in D$, then there exists an isomorphism
    
\[\pi_1^{N,log}(X,x_0)(k) \simeq \pi_1^{log,\'et}(X,\widetilde{x_0})^{(p')} \simeq  \pi_1^{t}(U,\underline{x_0})^{(p')}\]
 where $\pi_1^{t}$ is the Grothendieck-Murre fundamental group \cite{tame-fundamental} and $\underline{x_0}$ is the underlying geometric point of $\widetilde{x_0}$.
\end{remark}

\bibliographystyle{alpha} %amsplain}
\bibliography{prebib}
\end{document}